\documentclass[11pt]{amsart}

\usepackage {enumerate}
\usepackage{setspace}
\usepackage{caption}
\usepackage{mathrsfs}
\usepackage{hyperref}
\usepackage{esint}
\usepackage{amssymb}

\usepackage{color}

\newtheorem{theorem}{Theorem}[section]
\newtheorem{lemma}[theorem]{Lemma}
\newtheorem{corollary}[theorem]{Corollary}
\newtheorem{proposition}[theorem]{Proposition}

\numberwithin{equation}{section}

\theoremstyle {definition}

\newtheorem{remark}[theorem]{Remark}

\usepackage[margin=1.1in]{geometry}

\allowdisplaybreaks
  
\DeclareMathOperator{\area}{area} 
\DeclareMathOperator{\vol}{vol}
\newcommand{\R}{\mathbb{R}}
\DeclareMathOperator{\Ric}{Ric}

\title[Large isoperimetric regions in asymptotically S-AdS initial data]{Characterization of large isoperimetric regions in asymptotically hyperbolic initial data}

\author{Otis Chodosh}
\address{Department of Mathematics, Princeton University, Fine Hall, Washington Road, Princeton, NJ, 08544, United States}
\email{ochodosh@princeton.edu}

\author{Michael Eichmair}
\address{Faculty of Mathematics, University of Vienna, Oskar-Morgenstern-Platz 1, 1090 Vienna, Austria}
\email {michael.eichmair@univie.ac.at}

\author{Yuguang Shi}
\address{Key Laboratory of Pure and Applied Mathematics, School of Mathematical Sciences, Peking University, Beijing, 100871, P.~R.~China}
\email{ygshi@math.pku.edu.cn}

\author{Jintian Zhu}
\address{Key Laboratory of Pure and Applied Mathematics, School of Mathematical Sciences, Peking University, Beijing, 100871, P.~R.~China}
\email{zhujt@pku.edu.cn}

\begin{document}

\begin {abstract}
Let $(M,g)$ be a complete Riemannian $3$-manifold asymptotic to Schwarzschild-anti-deSitter and with scalar curvature $R \geq - 6$. Building on work of A.~Neves and G.~Tian and of the first-named author, we show that the leaves of the canonical foliation of $(M, g)$ are the unique solutions of the isoperimetric problem for their area. The assumption $R \geq -6$ is necessary.

This is the first characterization result for large isoperimetric regions in the asymptotically hyperbolic setting that does not assume exact rotational symmetry at infinity. 

\end {abstract}

\maketitle

\date{}


\section {Introduction}

The systematic study of stable constant mean curvature spheres in initial data sets for the Einstein equations has been pioneered in the work of D.~Christodoulou and S.-T.~Yau \cite{Christodoulou-Yau:1988} and of G.~Huisken and S.-T. Yau \cite{Huisken-Yau:1996}. The existence of the canonical foliation of the end of initial data asymptotic to Schwarzschild-anti-deSitter has been established by R.~Rigger \cite{Rigger:2004}. A.~Neves and G.~Tian \cite{Neves-Tian:2009, Neves-Tian:2010} have shown that the leaves of this foliation are the unique stable constant mean curvature spheres that enclose the center of the manifold and which satisfy a pinching condition that relates their inner and their outer radius. We refer the readers to Appendix \ref{sec:initialdata} for notation and to Appendix \ref{sec:hawkingalongfoliation} for a more detailed discussion of these results. 

In Theorem \ref{thm:main1}, we observe that the pinching condition used in \cite{Neves-Tian:2009}, stated here as \eqref{pinching}, may be replaced by an integral condition in the form of an a priori bound on their Hawking mass. 

\begin {theorem} \label{thm:main1}
Let $(M, g)$ be asymptotic to Schwarzschild-anti-deSitter with mass $m > 0$. Let $\Lambda > 0$. There is a constant $r_0 > 1$ with the following property. Every stable constant mean curvature sphere $\Sigma$ in $(M, g)$ that encloses $B_{r_0}$ and with Hawking mass $m_{H} (\Sigma) \leq \Lambda$ is a leaf of the canonical foliation.
\end {theorem}

Assume now that $R \geq -6$ where $R$ is the scalar curvature of $(M, g)$. The existence of isoperimetric surfaces in $(M, g)$ for every sufficiently large area has been proven by the first-named author \cite{Chodosh:2016}, together with a bound of their Hawking mass. In conjunction with ideas from \cite{Ji-Shi-Zhu:2015, CESY}, we obtain from Theorem \ref{thm:main1}  our second main result in this paper: \textit{A fully global characterization of the leaves of the canonical foliation as the unique large solutions of the isoperimetric problem.} 

\begin {theorem} \label{thm:main2}  
Let $(M, g)$ be a complete Riemannian $3$-manifold asymptotic to Schwarzschild-anti-deSitter with mass $m > 0$. We also assume that $R \geq -6$ and that $\partial M$ is connected and the only closed $H = 2$ surface in $(M, g)$. There is $V_0 > 1$ with the following property. Let $\Omega \subset M$ be an isoperimetric region of volume $V$ where $V \geq V_0$. Then $\Omega$ is bounded by $\partial M$ and a leaf of the canonical foliation. In particular, the solutions of the isoperimetric problem in $(M, g)$ for sufficiently large volumes are unique. 
\end {theorem}

When $(M,g)$ is exactly Schwarzschild-anti-deSitter, this result was proven by J.~Corvino, A.~Gerek, M.~Greenberg, and B.~Krummel in \cite{CorvinoGerekGreenbergKrummel} building on the pioneering work of H.~Bray \cite{Bray:1997}. When $(M, g)$ is \textit{isometric} to Schwarzschild-anti-deSitter outside of a compact set, Theorem \ref{thm:main2} was proven by the first-named author in \cite{Chodosh:2016}. It is shown in Section 10 of \cite{Chodosh:2016} that Theorem \ref{thm:main2} fails when the condition $R \geq -6$ is dropped: There exist rotationally symmetric $(M, g)$ with outermost $H = 2$ boundary that are equal to Schwarzschild-anti-deSitter outside of a compact set and in which \textit{no} large centered coordinate sphere is isoperimetric. Finally, we note that S.~Brendle has shown that in exact Schwarzschild-anti-deSitter (or Schwarzschild), the centered coordinate spheres are the unique embedded closed constant mean curvature surfaces \cite{Brendle:2013}. \\

We conclude with a brief account of the available results in the asymptotically flat setting. \\

The optimal, global uniqueness result for stable constant mean curvature spheres in initial data asymptotic to Schwarzschild has recently been established by the first- and the second-named authors in \cite{angstnomore, CE:far-outlying}, building on earlier work of G.~Huisken and S.-T.~Yau \cite{Huisken-Yau:1996}, of J.~Qing and G.~Tian \cite{Qing-Tian:2007}, of J.~Metzger and the second-named author \cite{stablePMT}, of S. Brendle and the second-named author \cite{Brendle-Eichmair:2014}, as well as that of A. Carlotto and the first- and second-named authors \cite{mineffectivePMT}. We refer to the introduction of \cite{angstnomore} for a comprehensive account and more detailed description of these and other important contributions in this context. \\

The global uniqueness of large isoperimetric surfaces in asymptotically flat manifolds with non-negative scalar curvature has been established recently in joint work \cite{CESY} of H.~Yu and the first-, second-, and third-named authors. Building on the work of H.\ Bray \cite{Bray:1997} for metrics which are exactly Schwarzschild outside of a compact set, global uniqueness of large solutions of the isoperimetric problem in $(M, g)$ asymptotic to Schwarzschild with mass $m > 0$ has been shown by J.~Metzger and the second-named author in \textit{any} dimension and with no assumption on the scalar curvature \cite{isostructure, hdiso}. These results in \cite{isostructure,hdiso,CESY} resolve a long-standing conjecture of G.~Huisken. \\

Finally, we note that there are very few geometries where we have complete understanding of the isoperimetric problem in the large. To our knowledge, the results in \cite{isostructure, hdiso, CESY} and Theorem 1.2 above are the only examples with no exact symmetries.
 \mbox{} \\

\noindent {\bf Acknowledgments.} Otis Chodosh is supported by the Oswald Veblen Fund and the NSF grant No.~1638352. Michael Eichmair is supported by the START-Project Y963-N35 of the Austrian Science Fund (FWF). Otis Chodosh and Michael Eichmair thank the Erwin Schr\"odinger Institute for its warm hospitality during the program \textit{Geometry and Relativity} in the Summer of 2017. They also thank Professors Hubert Bray, Simon Brendle, Gerhard Huisken, Jan Metzger, and Richard Schoen for their kind support and encouragement. Yuguang Shi and Jintian Zhu are supported by the NSFC grants No.~11671015 and 11731001. They would like thank Professor Li Yuxiang for his kind support.


\section {Proof of Theorem \ref{thm:main1}}

Throughout this section, we let $(M, g)$ be a Riemannian $3$-manifold that is asymptotic to Schwarzschild-anti-deSitter with mass $m > 0$. 

We assume that $\Sigma$ is a stable constant mean curvature sphere in $(M, g)$. The mean curvature of $\Sigma$ with respect to its outward pointing unit normal $\nu$ is denoted by $H$. We also assume that $\Sigma$ and $\partial M$ together bound a compact region $\Omega$ in $M$, and that $B_{r_0} \subset \Omega$ where $r_0 > 1$ is a large numerical constant that depends only on $(M, g)$. \\

Fix $\Lambda > 1$. We assume that  
\[
m_{H} (\Sigma) \leq \Lambda 
\]
where 
\begin{align} \label{eqn:Hawkingmassbound} 
m_{H} (\Sigma) = \frac{|\Sigma|^{\frac{1}{2}}}{ (16 \, \pi)^{\frac{3}{2}}} \left ( 16 \, \pi - (H^2 - 4) \,  |\Sigma|  \right)
\end{align}
is the Hawking mass of $\Sigma$. Note that \eqref{eqn:Hawkingmassbound} is  equivalent to either one of the bounds  
\begin{align} \label{boundH}
16 \, \pi - (H^2 - 4) |\Sigma| \leq O (|\Sigma|^{-1/2}) \qquad \text{ or } \qquad \frac{4 \, \pi}{|\Sigma|} \leq (H - 2) + O(|\Sigma|^{-3/2}).
\end{align}


\begin {lemma} We have
\begin {align} \label{ringh2CY} 
\int_\Sigma |\mathring h|^2 = O (1) \int_\Sigma e^{- 5 \, r} + O (|\Sigma|^{-1/2}).
\end{align}
\begin {proof} This follows from \eqref{CY} and $R + 6 = O (e ^{- 5 \, r})$. 
\end {proof}
\end {lemma}


\begin {lemma} 
We have 
\begin{align} \label{decaye3r}
\int_\Sigma e^{- 3 \, r} = O (|\Sigma|^{-1/2}).
\end{align}
\begin {proof}
Integrating \eqref{expansionGauss} and using the Gauss-Bonnet formula, we obtain 
\[
16 \, \pi - (H^2 - 4) \, |\Sigma|= - 2 \, \int_\Sigma |\mathring h|^2 + (8 \, m +o(1))  \int_\Sigma e^{- 3 \, r} .
\]
The estimate follows in conjunction with \eqref{ringh2CY}, using that $m > 0$.
\end {proof}
\end {lemma}


\begin {lemma} We have 
\begin{align} \label{goodHestimateHawking}
16 \, \pi - (H^2 - 4) \, |\Sigma| = O (|\Sigma|^{-1/2}) \qquad \text{ and } \qquad \frac{4 \, \pi}{|\Sigma|} = (H - 2) + O(|\Sigma|^{-3/2}).
\end{align}
\begin {proof}
Combining \eqref{CY} with \eqref{decaye3r}, we obtain the lower bound $m_{H} (\Sigma) \geq - o (1)$ as $r_0 (\Sigma) \to \infty$. 
\end {proof}
\end {lemma}


\begin {lemma} We have 
\begin{align} \label{improvedisoperimetric}
\int_\Sigma 4 \, e^{- 2 \, r} \geq 4 \, \pi + O (|\Sigma|^{-1/2}).
\end{align}
\begin{proof}
We use \eqref{decaye3r} to sharpen the steps leading to Proposition 4.2 (iii) in \cite{Neves-Tian:2009}:  We denote by $\Sigma_\delta \subset B_1(0)$ the surface with $\Psi (\Sigma_\delta) = \Sigma$ and by $\Omega_\delta$ the compact region enclosed by it. (We recall this notation in Appendix \ref{sec:initialdata}.) From \eqref{decaye3r}, we obtain 
\begin {align} \label{estimatingareadelta}
\int_\Sigma 4 \, e^{- 2 \, r} \, d \, \bar \mu = \area_\delta (\Sigma_\delta) + O (|\Sigma|^{-1/2}),
\end {align}
where we have also used \eqref{diskplaneaux} and 
\[
(1 + \cosh r)^{-2} = 4 \, e ^{- 2 \, r} + O (e^{- 3 \, r}).
\]
Moreover,  
\begin{align} \label{estimatingvolumdelta}
\vol_\delta (\Omega_\delta) = \frac{4 \, \pi}{3} - O (|\Sigma|^{-1/2}).
\end{align}
Indeed, 
\[
\frac{4 \, \pi}{3} = \vol_\delta (\Omega_\delta) + \vol_\delta (B_1(0) \setminus \Omega_\delta)
\]
and 
\[
 \vol_\delta (B_1(0) \setminus \Omega_\delta) = O(1) \int_{\Sigma_\delta} (1 - s^2) = O(1) \int_\Sigma e^{-3 \, r} = O (|\Sigma|^{-1/2}).
\]
The claim follows from the Euclidean isoperimetric inequality
\[
\area_\delta (\Sigma_\delta) \geq 4 \, \pi \, \left( \frac{3  \vol_\delta (\Omega_\delta) }{4 \, \pi}  \right)^{\frac{2}{3}}.\qedhere
\]
\end{proof}
\end {lemma}


\begin {corollary}
We have
\begin{align}  \label{improvedgradientdecay}
\int_\Sigma \left( 1 - \langle \nu, \nabla r\rangle \right)^2 &= O ( |\Sigma|^{-1/2} ) \\
\int_\Sigma e^{- 2 \, r} \, |\nabla_\Sigma r|^2 &= O ( |\Sigma|^{-1/2} ) \label{improvedgradientdecayII}\\
\int_\Sigma 4 \, e^{- 2 \, r} &= 4 \, \pi + O (|\Sigma|^{-1/2}). \label{eqn:4e2r}
\end{align}
\begin{proof}
We use our stronger estimates to sharpen the steps leading to Proposition 4.2 in \cite{Neves-Tian:2009}. 

Integrating \eqref{deltar}, but using \eqref{boundH} and \eqref{decaye3r} to bound the error, we obtain  
\begin{align} \label{aux:4e2r}
\int_\Sigma \left( 4 - 2 \, |\nabla_\Sigma r|^2\right) \, e^{- 2 \, r} + \frac{4 \, \pi }{|\Sigma|} \, \int_\Sigma \left( 1 - \langle \nu, \nabla r \rangle \right) + \int_\Sigma \left( 1 - \langle \nu, \nabla r \rangle \right)^2 = 4 \, \pi + O (|\Sigma|^{-1/2} ). 
\end{align}
Using \eqref{improvedisoperimetric}, we conclude
\[
- 2 \int_\Sigma |\nabla_\Sigma r|^2 \,  e^{- 2 \, r} +  \frac{4 \, \pi}{|\Sigma|} \,  \int_\Sigma \left( 1 - \langle \nu, \nabla r \rangle \right)  + \int_\Sigma \left( 1 - \langle \nu, \nabla r \rangle \right)^2 \leq O  (|\Sigma|^{-1/2} ).
\]
Using the pointwise estimates
\begin {align*}
\frac{4 \, \pi}{|\Sigma|} \, \left| 1 - \langle \nu, \nabla r \rangle \right| &\leq \frac{1}{4}\left( 1 - \langle \nu, \nabla r \rangle \right)^2 + O  ( |\Sigma|^{- 2} ) \\
2 \, e^{- 2 \, r} \,  |\nabla_\Sigma r|^2 &\leq \frac{1}{4} \left( 1 - \langle \nu, \nabla r \rangle \right)^2 + O (  e^{-4 \, r} )
\end{align*}
and \eqref{decaye3r}, we obtain \eqref{improvedgradientdecay}. Estimate \eqref{improvedgradientdecayII} follows from this and Cauchy-Schwarz. We obtain \eqref{eqn:4e2r} by using \eqref{improvedgradientdecay} and \eqref{improvedgradientdecayII} in  \eqref{aux:4e2r}.
\end{proof}
\end {corollary}


\begin {corollary} \label{cor:isoperimetrydelta} 
We have  
\begin{align*}
\area_\delta (\Sigma_\delta) &= 4 \, \pi + O (|\Sigma|^{-1/2}), \\
\vol_\delta (\Omega_\delta)  &= \frac{4 \, \pi}{3} + O (|\Sigma|^{-1/2})
\end{align*}
where we use the notation explained in Appendix \ref{sec:initialdata}.
\begin {proof} 
The Euclidean area estimate follows from \eqref{estimatingareadelta} together with \eqref{eqn:4e2r} and \eqref{decaye3r}. The Euclidean volume estimate is a restatement of \eqref{estimatingvolumdelta}.

\end {proof}
\end {corollary}


We now rescale $\Sigma$ homothetically, 
\begin{align} \label{eqn:hatg}
\hat g_\Sigma = \left( \sinh \hat r \right)^{-2} \, \bar g|_\Sigma
\end{align}
where $\hat r >$ is the hyperbolic \textit{area radius} 
\[
\area_{\bar g}(\Sigma) = 4 \, \pi \, (\sinh \hat r)^2.
\]
The same rescaling is studied by A.~Neves and G.~Tian in Section 5 of \cite{Neves-Tian:2009}. Instead of their pinching estimate \eqref{pinching}, we use our bound \eqref{eqn:Hawkingmassbound} on the Hawking mass to estimate the Gaussian curvature of $\hat g_\Sigma$ and to produce a good conformal parametrization. 


\begin {lemma} Let $p$ be such that $1 < p < 3/2$. As $\underline r(\Sigma) \to \infty$, 
\[
\| \hat K - 1 \|_{L^p(\hat \mu)} = o (1).
\]
\begin {proof}
First, by \eqref{expansionGauss}, 
\[
\bar K = K + O (e^{- 3 \, r}) = \frac{1}{4} \,  (H^2 - 4) - \frac{1}{2} \, |\mathring {h}|^2 + O (e^{- 3\, r}).
\]
In conjunction with \eqref{goodHestimateHawking}, we obtain
\begin {align*}
\hat K =  (\sinh \hat r )^2 \, \bar K 
&= (\sinh \hat r )^2 \left( \frac{1}{4}  \, (H^2 - 4)- \frac{1}{2} \, |\mathring {h}|^2 + O (e^{- 3 \, r} ) \right) \\
&= (\sinh \hat r )^2 \left( \frac{4 \, \pi}{|\Sigma|} - \frac{1}{2} \, |\mathring {h}|^2 + O (e^{- 3 \, \hat r}) + O (e^{- 3\, r}) \right) \\
&= 1+ O (e^{-  \hat r}) +  |\mathring h|^2 \, O (e^{2 \, \hat r }) + O (e^{2 \, \hat r}  e^{- 3 \, r}).
\end {align*}
The assertion follows from
\begin {align*}
\int_\Sigma (e^{2 \, \hat r} e^{- 3 \, r})^p \, d  \hat \mu &= O (  e^{2 \, (p - 1) \, \hat r}) \int_\Sigma e^{- 3 \, p \, r}  = O ( e^{2 \, (p - 1) \, \hat r}) \int_\Sigma e^{- 3 \, r} \\& =  O (e^{(2 \, p - 3) \,  \hat r})
\end{align*}
and 
\begin{align*}
\int_\Sigma e^{2 \, \hat r \, p} \,  |\mathring h|^{2 \, p} \, d  \hat \mu &= O (  e^{2 \, (p - 1) \, \hat r}) \int_\Sigma |\mathring h|^{2 \, p}   = o ( e^{2 \, (p - 1) \, \hat r}) \int_\Sigma |\mathring h|^{2}   \\ & = o ( e^{(2 \, p - 3) \, \hat r} )
\end {align*}
where we have used \eqref{eqn:secondffbound}. 
\end {proof}
\end {lemma}


The following is now an immediate consequence of Theorem A.1 in \cite{Nerz:2014-characterization}.

\begin {corollary} \label{cor:uniformizationhat}
Fix $p$ with $1 < p < 3/2$.  There is a diffeomorphism $\hat \varphi : \mathbb{S}^2 \to \Sigma$ with 
\[
\hat \varphi^* \hat g_\Sigma = e^{2 \, \hat \beta} g_{\mathbb{S}^2}
\]
where, as $\underline r (\Sigma) \to \infty$, 
\begin{align} \label{estimatehatbeta}
\|\hat \beta\|_{W^{2, p} (\mathbb{S}^2)} = o (1).
\end{align}
\end {corollary}

We also consider the conformal rescaling 
\begin{align} \label{eqn:tildeg}
\tilde g_\Sigma = \psi^{-2} \bar g|_\Sigma
\end{align}
where 
\[
\psi = 2 \, \cosh^2 \frac{r}{2} = 1 + \cosh r. 
\]

Note that $\Sigma$ with the Riemannian inner product $\tilde g_\Sigma$ is isometric to the Euclidean surface $\Sigma_\delta \subset B_1(0)$. Here we use the notation explained in Appendix \ref{sec:initialdata}.

The conformal rescaling \eqref{eqn:tildeg} is also considered by A.~Neves and G.~Tian in Section 6 of \cite{Neves-Tian:2010}. They use it in conjunction with a result of C.~De Lellis and S.~M\"uller \cite{DeLellis-Muller:2005} to show that $\Sigma$ is close to a coordinate sphere in the chart at infinity. In Proposition \ref{prop:tildeg} below, we apply results from \cite{DeLellis-Muller:2005} rather differently to obtain a suitable conformal parametrization of $\tilde g_\Sigma$.  


\begin {proposition} \label{prop:tildeg}
There is a diffeomorphism $\tilde \varphi : \mathbb{S}^2 \to \Sigma$ with 
\[
 \tilde \varphi^* \tilde g_\Sigma = e^{2 \, \tilde \beta} g_{\mathbb{S}^2}
 \]
where
 \begin{align} \label{estimatetildebeta}
 \|\tilde \beta\|_{L^\infty}^2 + \|\tilde \beta\|^2_{W^{1, 2} (\mathbb{S}^2)} = O ( |\Sigma|^{-1/2} ).
\end{align}
\begin {proof}
We have that 
\[
\bar g_\Sigma =  g_\Sigma + O (e^{- 3 \, r}) \qquad \text{ and } \qquad \bar h =  h + O (e^{- 3 \, r}). 
\]
Here we use the bound 
\[
\sup_{x \in \Sigma} |h(x)| = O (1)
\]
from \eqref{eqn:secondffbound} to obtain the second estimate. Using  also \eqref{ringh2CY} and \eqref{decaye3r}, we find
\[
 \int_\Sigma |\overline {\mathring h}|_{\bar g}^2 \, d  \bar \mu = \int_\Sigma |\mathring h|^2 + O ( |\Sigma|^{-1/2} ) = O (|\Sigma|^{-1/2}). 
\]
By conformal invariance,  
\[
\int_{\Sigma_\delta} |\mathring h_\delta|^2 \, d  \mu_\delta =  \int_\Sigma |\overline {\mathring h}|_{\bar g}^2 \, d  \bar \mu= O ( |\Sigma|^{-1/2} ). 
\]
By Corollary \ref{cor:isoperimetrydelta}, 
\[
\area_\delta (\Sigma_\delta) = 4 \, \pi + O ( |\Sigma|^{-1/2} ). 
\]
The result now follows from Proposition 3.2 in \cite{DeLellis-Muller:2005}. 
\end {proof}
\end {proposition}


We recall from \cite[p.~929]{Neves-Tian:2009}, cf.~\cite[Lemma 1]{Chang:1996}, the definition of the functional 
\[
S (u) = \int_{\mathbb{S}^2} |\nabla_{\mathbb{S}^2} u |^2 - 2 \int_{\mathbb{S}^2} u
\]
and its conformal invariance: 
\[
S(u) = S(v)
\]
whenever $\psi: \mathbb{S}^2 \to \mathbb{S}^2$ is a conformal diffeomorphism and  
\[
\psi^* (e^{- 2 \, u} g_{\mathbb{S}^2})= e^{- 2 \, v} g_{\mathbb{S}^2}. 
\]

A straightforward computation gives that   
\begin{align} \label{defxi}
e^{- 2 \, \xi} \, e^{2 \, \hat \beta} g_{\mathbb{S}^2} = (\tilde \varphi^{-1} \circ \hat \varphi)^* (e^{2 \, \tilde \beta} g_{\mathbb{S}^2})
\end{align}
where 
\[
\xi = w- \log (1 - e^{- 2 \, \hat r}) + \log (1 - e^{- 2 \, (r \circ \hat \varphi)}) + \log \coth \frac{r \circ \hat \varphi}{2} 
\]
and 
\[
w = (r \circ \hat \varphi) - \hat r.
\]
Note that the conclusion of Theorem \ref{thm:main1} is equivalent to the bound 
\[
w = O (1). 
\]
We establish this bound in three strides. 


\begin {lemma} \label{lem:Sw}
As $\underline r (\Sigma) \to \infty$,  
\begin{align} \label{swbounded}
S(w) = o (1).
\end{align}

\begin {proof}
In view of \eqref{defxi}, we have 
\[
S (\xi) = S( - \tilde \beta + \hat \beta \circ \psi^ {-1})
\]
where $\psi = \tilde \varphi^{-1} \circ \hat \varphi  : \mathbb{S}^2 \to \mathbb{S}^2$. Using conformal invariance and \eqref{improvedgradientdecayII}, we get 
\[
\int_{\mathbb S^2} e^{- 2 \, (r \circ \hat \varphi)} |\nabla_{\mathbb{S}^2} (r \circ \hat \varphi) |^2 = o (1),
\] 
which in turn implies 
\[
\int_{\mathbb{S}^2} |\nabla_{\mathbb{S}^2} (\xi - w)|^2 = o(1).
\]
Since 
\[
\|\xi - w\|_{L^\infty (\Sigma)} = o (1),
\]
we obtain  
\[
S (w) = S(\xi) + o (1).
\]
Using conformal invariance of energy as well as estimates \eqref{estimatehatbeta} and \eqref{estimatetildebeta}, we find  
\[
\int_{\mathbb{S}^2} |\nabla_{\mathbb{S}^2} (- \tilde \beta + \hat \beta \circ \psi^ {-1})|^2 = o (1). 
\] 
Finally, the estimate 
\[
\|  - \tilde \beta + \hat \beta \circ \psi^ {-1}\|_{L^\infty(\mathbb{S}^2)} = o (1)
\]
follows from \eqref{estimatehatbeta} and \eqref{estimatetildebeta}. Putting these estimates together, we obtain \eqref{swbounded}.
\end{proof}
\end {lemma}


\begin {lemma} \label{lem:Deltawp}
Fix $p$ with $1 < p < 3/2$. Then
\[
\| \Delta_{\mathbb{S}^2} w \|_{L^p (\mathbb{S}^2)} = O (1).
\]
\begin{proof}
Note that 
\[
\Delta_\Sigma r = \Delta_{\bar g|_{\Sigma}} r + O (e^{- 3 \, r}).
\]
Combining \eqref{deltar} and \eqref{goodHestimateHawking}, we deduce 
\begin{align} \label{auxauxaux}
\Delta_{\bar g|_{\Sigma}} r & = (4 - 2 |\nabla_\Sigma r|^2) e^{- 2 \, r} - \frac{4 \, \pi}{|\Sigma|} + \frac{4 \, \pi}{|\Sigma|} \, (1 - \langle \nu, \nabla r \rangle) + (1 - \langle \nu, \nabla r \rangle)^2 \nonumber \\
& \qquad + O (e^{- 3 \, r}) + O (e^{- 3 \, \hat r}). \nonumber \\
& = O (e^{- 2 \, r}) + O (e^{- 2 \, \hat r}) + (1 - \langle \nu, \nabla r \rangle)^2.
\end{align}
From the conformal invariance of the Laplace-Beltrami operator on surfaces, we obtain 
\begin{align*}
\left( \Delta_{\bar g |_\Sigma} r \right) \circ \hat \varphi = \Delta_{\hat \varphi^*  \bar g|_\Sigma} w = \Delta_{ ( \sinh \hat r )^2 \, e^{2 \,  \hat \beta} \, g_{\mathbb{S}^2}} w = (\sinh \hat r)^{-2} \, e^{- 2 \, \hat \beta} \,  \Delta_{\mathbb{S}^2} w.
\end{align*}
From this, we verify the asserted bound term by term. First, note that 
\[
e^{2 \, \hat \beta} = 1 + o (1)
\]
by \eqref{estimatehatbeta} and Sobolev embedding. Using \eqref{decaye3r} and H\"older's inequality, we obtain  
\[
\int_{\mathbb{S}^2} e ^{- 2 \, p \, (r \circ \hat \varphi)} \,  e ^{2  \, p \, \hat r } = O ( e^{2 \, p \, \hat r} \, e^{- 2 \, \hat r}) \int_\Sigma e^{- 2 \, p \, r}  = O (1). 
\]
This bounds the contribution to $\| \Delta_{\mathbb{S}^2} w \|_{L^p (\mathbb{S}^2)}$ from the first term in \eqref{auxauxaux}. To estimate the contribution from the third term, we apply \eqref{improvedgradientdecay} to obtain 
\begin {align*}
\int_\Sigma (1 - \langle \nu, \nabla r \rangle)^{2 \, p} \leq \int_\Sigma (1 - \langle \nu, \nabla r \rangle)^2 = O ( e^{- \hat r} ).
\end {align*}
Here we also use that $\langle \nu, \nabla r \rangle > 0$ by Lemma \ref{lem:tangentialgradiento1}. The bounds for the other terms follow from these. 
\end {proof}
\end {lemma}


\begin {proposition} \label{prop:w2pw}
Let $p$ be such that $1 < p < 3/2$. Then  
\begin{align} \label{eqn:w2pw}
\|w\|_{W^{2, p} (\mathbb{S}^2)} = O (1).
\end{align}
In particular, $w$ is bounded.
\begin {proof}
By Lemma \ref{lem:Sw} and Lemma \ref{lem:Deltawp}, we have that 
\[
|S(w)| = O(1) \qquad \text{ and } \qquad \|\Delta_{\mathbb{S}^2} w\|_{L^p(\mathbb{S}^2)} = O(1).
\]
Let 
\[
f = \Delta_{\mathbb{S}^2} w \qquad \text{ and } \qquad v = w - \fint_{\mathbb{S}^2}w.
\]
Testing the equation 
\[
\Delta_{\mathbb{S}^2} v = f
\]
with $v$ and using Cauchy-Schwarz and the Poincar\'e inequality, we obtain 
\[
\|\nabla_{\mathbb{S}^2} v\|^2_{L^2 (\mathbb{S}^2)} \leq \| f \|_{L^p (\mathbb{S}^2)} \, \| v\|_{L^q (\mathbb{S}^2)} \leq O (1) \|\nabla_{\mathbb{S}^2} v\|_{L^2 (\mathbb{S}^2)}.
\]
It follows that  
\[
 \|\nabla w\|_{L^2(\mathbb{S}^2)} = O (1). 
\]
In conjunction with $S(w) = O(1)$, we find 
\[
\fint_{\mathbb{S}^2} w = O (1). 
\]
Putting these estimates together, we obtain 
\[
\|w\|_{L^2(\mathbb{S}^2)} = O (1)
\]
from the Poincar\'e inequality. Standard elliptic theory now gives \eqref{eqn:w2pw}. 
\end {proof}
\end {proposition}

Since $w$ is bounded, the pinching condition \eqref{pinching} is satisfied, and Theorem \ref{thm:main1}  follows from the uniqueness results in \cite{Neves-Tian:2009}.
 

\section {Proof of Theorem \ref{thm:main2}}

In this section, we assume that $(M, g)$ is a complete Riemannian $3$-manifold asymptotic to Schwarzschild-anti-deSitter with mass $m > 0$, that $R \geq -6$, and that $\partial M$ is connected and the only closed $H = 2$ surface in $(M, g)$. \\

We need a large amount of area to bound a large amount of volume in $(M, g)$. This follows by comparison with hyperbolic space in the chart at infinity together with cut-and-paste arguments. Comparison with hyperbolic space also gives that  
\[
\lim_{A \to \infty} V_g(A) = \infty.
\]
We may thus use either area or volume to specify \textit{large} solutions to the isoperimetric problem. \\

We recall from Lemma \ref{lem:onelargecomponent} that every \textit{large enough} isoperimetric region has a \textit{unique large} component. The residue is a small collar about $\partial M$ by Lemma \ref{lem:smallcomponents}. \\

The crucial ingredient for the proof of Theorem \ref{thm:main2} is the characterization of \textit{isoperimetric spheres} as leaves of the canonical foliation given in the following proposition. 

\begin {proposition} \label{prop:largeisosphere}
There is $A_0 > 1$ with the following property. Let $\Omega$ be the unique large component of an isoperimetric region for area $A \geq A_0$. Its outer boundary $\Sigma = (\partial \Omega) \setminus (\partial M)$ is connected. If $\Sigma$ is a sphere, then it is a leaf of the canonical foliation and $\Omega$ coincides with the original isoperimetric region.  
\begin {proof} 
The connectedness of $\Sigma$ follows from the discussion in Appendix \ref{sec:isoperimetricprofile}.   
By Lemma \ref{lem:nodrift} and Lemma \ref{lem:mHbounds}, there is $r_0 > 1$ such that $\Omega \cap B_{r_0} \neq \emptyset$ and $m_H (\Sigma) \leq 4 \, m$ provided that $V > 0$ is sufficiently large. Taking $r_0 > 1$ larger, if necessary, Theorem \ref{thm:main1} shows that $\Sigma$ is a leaf of the canonical foliation if  $B_{r_0} \subset \Omega$. Theorem 4.3 from \cite{Ji-Shi-Zhu:2015} rules out the scenario $\Sigma \cap B_{r_0} \neq \emptyset$. Finally, Lemma \ref{lem:smallcomponents} shows that $\Sigma$ and $\partial M$ bound the original isoperimetric region. 
\end {proof}
\end {proposition}

We can now complete the proof of Theorem \ref{thm:main2} following the strategy of Section 9 in \cite{Chodosh:2016}, which in turn develops an idea of H.~Bray \cite{Bray:1997}. We sketch the full argument from \cite{Chodosh:2016}, where $(M, g)$ is assumed to be exactly Schwarzschild-anti-deSitter outside of a compact set, below, including the minor but necessary adaptations to our more general setting. Since the strategy is technical, we begin with an outline. \\ 

Let  $\{\Sigma_A\}_{A >  A_0}$ be the canonical foliation of $(M, g)$. We use $\Omega_A$ to denote the compact region bounded by $\Sigma_A$ together with $\partial M$. 

The derivative of 
\[
A \mapsto \vol_g( \Omega_A)
\]
is  the inverse of the mean curvature $H_A$ of $\Sigma_A$. We obtain an explicit estimate for $H_A$ from the expansion  
\begin{align} \label{aux:Hawkingoutline}
m_H(\Sigma_A) = \frac{A^{\frac{1}{2}}}{ (16 \, \pi)^{\frac{3}{2}}} \left( 16 \, \pi + 4 \, A - H_A^2 \, A \right) = m + O (A^{-1})
\end{align}
of the Hawking mass along the canonical foliation discussed in Lemma \ref{lem:Hacf}. 
Assume now that $I \subset (A_0, \infty)$ is an open interval such that, for every $A \in I$, there is an isoperimetric region the boundary of whose unique large component has non-zero genus. The derivative of the isoperimetric profile on such an interval $I$ is appreciably smaller than anticipated by \eqref{aux:Hawkingoutline}. Integrating up and comparing with the volume enclosed by the centered coordinate sphere $S_{r(A)}$ of area $A$, 
\begin{align} \label{eqn:isoballs}
\vol_g(B_{r(A)}) = \frac{1}{2} \, A - \pi \, \log A + O (1), 
\end{align}
cf.~Lemma \ref{lem:isoballs}, it follows that the interval $I$ is bounded. From this, we conclude that a \textit{maximally extended} such interval $I$ \textit{far out} has the form $(A_1, A_2)$ where $\Omega_{A_1}$, $\Omega_{A_2}$ are isoperimetric. We can rule out the existence of such intervals by studying the derivative of the isoperimetric profile at the endpoins $A_1, A_2$. \\

We now proceed to make this argument precise. \\

It follows from standard compactness properties of isoperimetric regions that   
\[
\mathcal{I} = \{ A > A_0 : \Omega_A \text{ is \textit{not} isoperimetric}\}
\]
is open. Let $A \in \mathcal{I}$.  By Proposition \ref{prop:largeisosphere}, the unique large component of an isoperimetric region for area $A$ has boundary of non-zero genus. Using this input, we estimate the derivative of the isoperimetric profile on connected components of $\mathcal{I}$. 

\begin {lemma} [Cf.~Lemma 9.1 and Proposition 9.3  in \cite{Chodosh:2016}] \label{lem:torusintervals}
There is $A_0 > 1$ with the following property. Let $(A_1, A_2) \subset [A_0, \infty)$ be such that, for every $A \in (A_1, A_2)$, there exists an isoperimetric region for area $A$ and the boundary of whose unique large component has non-zero genus. Then 
\begin{align} \label{eqn:torusintervals}
(V_g (A_2)_{-}')^{-2}  - 23 \, \pi \, A_2^{-1} \leq (V_g(A_1)_+') ^{-2} - 23 \, \pi \, A_1^{-1}. 
\end{align}
\begin {proof}
Assume first that the isoperimetric profile is smooth on $[A_1, A_2]$. Fix $A \in (A_1, A_2)$ and let $\Sigma$ be the boundary of the unique large component of an isoperimetric region whose boundary has area $A$. Following the proof of Lemma 9.1 in \cite{Chodosh:2016}, using Lemma \ref{lem:mHbounds} instead of Proposition 8.3 in \cite{Chodosh:2016}, we obtain   
\begin{align} \label{auxisoderivative}
2 \, V_g'' (A) \, V_g'(A)^{-3} \, (A - c)^2 \geq 24 \, \pi + \int_\Sigma (R + 6 + |\mathring h|^2) - 6 \, (A - c)^{- \frac{1}{2}} \, (16 \, \pi)^{\frac{3}{2}} \, m
\end{align}
where 
\[
c = A -  \area_g (\Sigma) = O (1).
\]
Recall from the discussion in Appendix \ref{sec:isoperimetricprofile} that the isoperimetric profile is increasing. Dropping the non-negative second term on the right-hand side and absorbing the third term into the first, we arrive at  
\[
2 \, V_g'' (A) \, V_g'(A)^{-3} \, A^2 \geq 23 \, \pi.
\]
Equivalently, the function 
\[
A \mapsto V_g' (A)^{-2} - 23 \, \pi \, A^{-1}
\]
is non-increasing on $(A_1, A_2)$. This gives \eqref{eqn:torusintervals} in the special case when the isoperimetric profile is smooth. In the general case, we can argue using weak derivatives exactly as in the proof of Proposition 6.3 in \cite{Chodosh:2016} to arrive at the same conclusion.
\end {proof}
\end {lemma}

\begin {lemma} [Cf.~Proposition 3.3 in \cite{Chodosh:2016}] \label{torusestimateisoperimetric}
For every $A_0 > 1$ there is $A \geq A_0$ with the following property. There does not exist an isoperimetric region for area $A$ such that the boundary of its unique large component has non-zero genus. 
\begin {proof}
Assume that the conclusion fails. It follows from Lemma \ref{lem:torusintervals} that there is $A_0 > 1$ large such that  
\begin{align} \label{auxxx}
4 \leq ( V_g(A)_+')^{-2} - 23 \, \pi \, A^{-1}
\end{align}
for \textit{all} $A \geq A_0$. Indeed, we have  \eqref{auxxx}  for all $A_1, A_2$ with $A_0 < A_1 < A_2$. We now take $A_2 \to \infty$ in \eqref{eqn:torusintervals} and use \eqref{eqn:convexity} and \eqref{eqn:largeisoH} to conclude \eqref{auxxx}. Using that $A \mapsto V_g(A)$ is strictly increasing and absolutely continuous, we conclude that 
\[
V_g(A) \leq \frac{1}{2} \, A  - \frac{23}{16} \, \pi \, \log A + O (1)
\]
for all $A \geq A_1$. This estimate contradicts \eqref{eqn:isoballs}. 
\end {proof}
\end {lemma}

\begin {lemma} [\protect{Cf.~\cite[p.~433]{Chodosh:2016}}] \label{lem:componentsno}
There can be no $A_1, A_2$ as in Lemma \ref{lem:torusintervals} such that the leaves $\Sigma_{A_1}$ and $\Sigma_{A_2}$ of the canonical foliation both bound isoperimetric regions.
\begin{proof}
Assume, for a contradiction, that such $A_1 < A_2$ exist. Then  
\[
H_{A_2}^{2} - 23 \, \pi \, A_2^{-1} \leq (V_g (A_2)_{-} ')^{-2}  - 23 \, \pi \, A_2^{-1} \leq (V_g (A_1)'_{+})^{-2}  - 23 \, \pi \, A_1^{-1} \leq H_{A_1}^{2} - 23 \, \pi \, A_1^{-1}
\]
from \eqref{eqn:convexity} and \eqref{eqn:torusintervals}. This, however, contradicts \eqref{HestimateNTspheres}.
\end{proof}
\end {lemma}

\begin {proof} [Proof of Theorem \ref{thm:main2}]

Lemma \ref{torusestimateisoperimetric} shows that $\mathcal{I} \subset (A_0, \infty)$ is not connected at infinity. On the other hand, Lemma \ref{lem:componentsno} gives that $\mathcal{I}$ has no bounded components provided that $A_0 > 0$ is sufficiently large. It follows that $\mathcal{I}$ is empty as long as $A_{0}>1$ is taken sufficiently large. Thus every leaf of the canonical foliation $\Sigma_A$ bounds an isoperimetric region. Thus 
\[
V_g(A) = \vol_g (\Omega_A)
\] 
for all $A > A_0$. In particular, the isoperimetric profile is a smooth function on $(A_0, \infty)$. Using the estimates for the lapse function of the canonical foliation and the geometry of the leaves from Section \ref{sec:hawkingalongfoliation}, we compute that
\[
2 \, V_g''(A) \, V_g'(A)^{-3} \, A^2 = 16 \, \pi + o (1)
\]
as $A \to \infty$. Assume that there exists \textit{another} isoperimetric region $\tilde \Omega_A$ for area $A > A_0$. We know from Proposition \ref{prop:largeisosphere} that the boundary of its unique large component has non-zero genus. From \eqref{auxisoderivative}, we obtain the estimate 
\[
2 \, V_g''(A) \, V_g'(A)^{-3} \, A^2 \geq 24 \, \pi + o (1). 
\]
This contradiction shows that $\Omega_A$ is the \textit{unique} isoperimetric region for area $A$.
\end{proof}


\appendix

\section {Asymptotically hyperbolic initial data} \label{sec:initialdata}

Just as A.~Neves and G.~Tian do in \cite{Neves-Tian:2009}, we work with two different standard models for three-dimensional hyperbolic space. We use $\bar g$ to denote the hyperbolic metric on $\R^3$ given by
\[
\bar g = d \, r\otimes d \, r + \sinh^2 r \, g_{\mathbb{S}^2}
\]
in polar coordinates. We will also use the disk model for hyperbolic space with metric tensor 
\[
\frac{4}{(1 - s^2)^2} \left( d\, s \otimes d \, s + s^2 g_{\mathbb{S}^2} \right) 
\]
in polar coordinates on $B_1(0)$. The radial map  
\[
s \mapsto r (s) = \log \frac{1 + s}{1 - s}
\]
induces an isometry $\Psi : B_1(0) \to \R^3$ between these models. In particular, 
\begin{align} \label{diskplaneaux}
\Psi^* \left( (1 + \cosh r)^{-2} \,  \bar g \right) = d\, s \otimes d \, s + s^2 g_{\mathbb{S}^2}.
\end{align}

When $\Sigma \subset \R^3$ is a surface, we use $\Sigma_\delta \subset B_1(0)$ to denote the Euclidean surface with 
\[
\Psi(\Sigma_\delta) = \Sigma.
\]

We say that a Riemannian $3$-manifold $(M, g)$ is \textit{asymptotic to Schwarzschild-anti-deSitter of mass $m > 0$} if it is connected and if there are a bounded open set $U \subset M$ and a diffeomorphism   
\[
M \setminus U \cong_x \R^3 \setminus B_1(0)
\]
such that, in polar coordinates, 
\begin{align} \label{def:SadS}
g = dr\otimes dr + \left( \sinh^2 r + \frac{2 m }{ 3 \sinh r}  \right) g_{\mathbb{S}^2} + Q
\end{align}
where   
\[
| Q  |_{\bar g} + | \bar \nabla Q  |_{\bar g} +  |  \bar \nabla^{2} Q |_{\bar g} = O ( e^{- 5 \, r} ).
\]
Note that 
\[
R + 6 = O (e^{- 5 \, r})
\]
where $R$ is the scalar curvature of $(M, g)$. 

Our convention here differs from that used in \cite{Neves-Tian:2009} by a factor of $2$ for the mass. 

We usually require \textit{in addition} that $(M, g)$ is complete and such that $\partial M$ is connected and the only closed $H = 2$ surface in $(M, g)$. It can be shown\footnote{This follows exactly as in the asymptotically flat case, cf.~e.g~Section 4 in \cite{Huisken-Ilmanen:2001}. The area functional is replaced by the appropriate brane functional as in \cite[Proposition 3.1]{Chodosh:2016}.} that, in this case, $M$ itself is diffeomorphic to $\mathbb{R}^{3}\setminus B_{1}(0)$.

Of course, Schwarzschild-anti-deSitter initial data itself satisfies these conditions. We recall that a closed form of the Schwarzschild-anti-deSitter metric (with boundary $H=2$) is given by 
\[
M = \{x \in \R^3 : |x| \geq 2 \, m\} \qquad \text{ and } \qquad g = \frac{1}{1+s^{2}-2 \, m \, s^{-1}} d \, s\otimes d \, s + s^{2} \, g_{\mathbb{S}^{2}}.
\]

We recall in passing that $(M, g)$ is said to be \textit{asymptotically hyperbolic} if, in place of  \eqref{def:SadS}, we have that 
\[
g = \bar g + P \qquad \text{ where } \qquad | P  |_{\bar g} + | \bar \nabla P  |_{\bar g} +  |  \bar \nabla^{2} P |_{\bar g} = O ( e^{- 3 \, r} ).
\]

We use $S_r \subset M$ to denote the image of the centered coordinate sphere $S_r(0)$ under this diffeomorphism, and let $B_r$ be the bounded component of $M \setminus S_r$. 

Let $\Sigma \subset M$ be a closed surface. We let 
\[
\underline r (\Sigma) = \sup \{ r > 1 : \Sigma \text{ encloses } B_r\} \qquad \text{ and } \qquad \overline r (\Sigma) = \inf \{ r > 1 : \Sigma \subset B_r\}.
\]


\section {Estimates for stable CMC spheres}

In this section, we recall several estimates for stable constant mean curvature surfaces $\Sigma$ in Riemannian $3$-manifolds $(M, g)$ that are used throughout the paper. Let $H$ denote the mean curvature of $\Sigma$ with respect to the (designated or natural) outward pointing unit normal $\nu$. \\

The \textit{Christodoulou-Yau estimate} for stable constant mean curvature spheres stated as \eqref{CY}  below is derived in \cite{Christodoulou-Yau:1988}. The proof of the weaker estimate \eqref{CYR} for surfaces of non-zero genus follows the same lines, using in addition the Brill-Noether theorem exactly as in the proof of Theorem 6 in \cite{Ritore-Ros:1992}.

\begin {lemma} [Cf.~ \cite{Christodoulou-Yau:1988, Ritore-Ros:1992}]
We have  
\begin{align} \label{CYR}
\frac{2}{3} \int_\Sigma |\mathring h|^2 + \frac{2}{3} \int_\Sigma (R + 6) + \int_\Sigma (H^2 - 4)  \leq \frac{64}{3} \, \pi.
\end{align}
The bound on the right-hand side can be sharpened to $16 \, \pi$ if $\Sigma$ has genus zero, so  in this case   
\begin{align} \label{CY}
\frac{2}{3} \int_\Sigma |\mathring h|^2 + \frac{2}{3} \int_{\Sigma} (R + 6)  \leq  16 \, \pi - \int_\Sigma (H^2 - 4) = 16 \, \pi - (H^2 - 4) \,  |\Sigma|.
\end{align}
\end {lemma} 

For the remaining results included in this section, we assume that $(M, g)$ is asymptotic to Schwarzschild-anti-deSitter with mass $m > 0$, that $\Sigma$ is a stable constant mean curvature  \textit{sphere}, and that $\Sigma$ encloses the centered coordinate ball $B_{2}$. 

\begin {lemma} [Cf.~Proposition 3.4 in \cite{Neves-Tian:2009}] We have  
\begin{equation} \label{deltar}
\Delta_\Sigma r = (4 - 2 \, |\nabla_\Sigma r|^2) e^{- 2 \, r} - (H - 2) + (H - 2) \, (1 - \langle \nu, \nabla r \rangle) + (1 - \langle \nu, \nabla r \rangle)^2 + O (e^{- 3 \, r}).
\end{equation}
\end {lemma}

\begin {lemma} [Cf.~Proposition 4.2 in \cite{Neves-Tian:2009}] There are constants $\lambda > 1$ and $r_0 > 1$ that depend only on $(M, g)$ with the following property. Assume that $\underline r (\Sigma) \geq r_0$. Then     
\begin {align*}
\frac{1}{\lambda} \, e^{2 \, \underline r (\Sigma)} & \leq |\Sigma| \leq \lambda \, e^{2 \, \overline r (\Sigma)} \\
\int_\Sigma \left( 1 - \langle \nu, \nabla r \rangle \right)^2 & \leq \lambda \, e^{-  \, \underline r (\Sigma)}
\end {align*}
and, for every integer $k \geq 1$, 
\begin {align*}
k \int_\Sigma |\nabla_\Sigma r|^2 \, e^{- k \, r} &\leq \lambda \, e^{- k \, \underline {r}(\Sigma)}.
\end{align*}
\end {lemma}

\begin {lemma} [Cf.~e.g.~Lemma 5.2 in \cite{Neves-Tian:2009}] 
As $\underline {r} (\Sigma) \to \infty$, 
\begin{align} \label{eqn:secondffbound} 
\sup_{x \in \Sigma} |\mathring h (x)| = o (1).
\end{align}
\end {lemma}

\begin {lemma} [Cf.~Lemma 6.7 in \cite{Neves-Tian:2009}] \label{lem:tangentialgradiento1}  
As $\underline {r} (\Sigma) \to \infty$, 
\begin{align*} 
\sup_\Sigma |\nabla_\Sigma r| = o (1).
\end{align*}
\end {lemma}

\begin {lemma} [Cf.~Lemma 3.2 in \cite{Neves-Tian:2009}]  
Let $K$ be the Gauss curvature of $\Sigma$. Then
\begin{align} \label{expansionGauss}
4 \, K = (H^2 - 4) - 2 \, |\mathring {h}|^2 + 64 \, m \, e^{- 3 \, r}  -  96 \, m \, |\nabla_\Sigma r|^2 e^{- 3 \, r} + O (e^{- 5\, r}).
\end{align}
\end {lemma}


\section {Canonical foliation} \label{sec:hawkingalongfoliation}

Let $(M, g)$ be asymptotic to Schwarzschild-anti-deSitter of mass $m > 0$. \\

It has been shown by R.~Rigger \cite{Rigger:2004} that there is a family of stable constant mean curvature spheres 
\begin{align} \label{canonicalfoliation}
\{\Sigma_A\}_{A > A_0} \qquad \text{ where } \qquad A = \area(\Sigma_A)
\end{align}
that foliate the complement of a compact subset of $M$. A. Neves and G. Tian \cite{Neves-Tian:2009} have shown that \textit{every} stable constant mean curvature sphere $\Sigma$ in $(M, g)$ that encloses $B_{r_0}$ and with  
\begin{align} \label{pinching}
\sup_{x \in \Sigma } r(x) - \frac{6}{5} \inf_{x \in \Sigma} r (x)  < - C
\end{align}
is a leaf of this \textit{canonical foliation}. Here, $r_0>1$ and $C>0$ are constants depending only on $(M, g)$. They also give an alternative proof of the existence of the canonical foliation. \\

We will use some of the estimates from \cite{Neves-Tian:2009} to estimate the Hawking mass along the foliation. 

For every $r > 0$ sufficiently large, consider the leaf $\Sigma$ of the canonical foliation \cite{Neves-Tian:2009} with mean curvature 
\[
H_m (r) = \frac{d}{d\, r} \log \left( \sinh^2 r + \frac{2 \, m}{3 \, \sinh r} + O (e^{- 2 \, r})\right) = 2 + 4\,  e^{- 2 \, r} - 16 \, m \, e^{- 3 \, r} + 4\,  e^{- 4 \, r} + O (e^{- 5 \, r}).
\]
From the estimates obtained in Section 8 of \cite{Neves-Tian:2009}, we see that  
\begin{align*}
\Ric (\nu, \nu) &= - 2 - 16 \, m \, e^{- 3 \, r} + O (e^{- 5 \, r}) \\
\mathring h &= O (e^{- 3 \, r}). 
\end{align*}
In particular, 
\[
\Ric (\nu, \nu) + |h|^2 = 8 \, e^{- 2 \, r} - 48 \, m \, e^{- 3 \, r} + 16 \, e^{- 4 \, r} + O (e^{- 5 \, r}),
\]
since clearly 
\[
H^2 =  4 + 16 \, e^{- 2 \, r}  - 64 \, m \, e^{- 3 \, r} + 32 \, e^{- 4 \, r} + O (e^{- 5 \, r}).
\]
Also, by the Gauss equation,
\[
2 \, K = 8 \, e^{- 2 \, r} + 16 \, e^{- 4 \, r} + O (e^{- 5 \, r}), 
\]
so that
\[
H^2 - 4 = 4 \, K - 64 \, m \, e^{- 3 \, r} + O (e^{- 5 \, r}).
\]
From this, we obtain the estimate 
\[
\frac{ |\Sigma|^{\frac{1}{2}} }{(16 \, \pi)^{\frac{3}{2}}} \Big(16 \, \pi - \int_\Sigma (H^2 - 4) \Big) = m + O (e^{- 2 \, r})
\]
for the Hawking mass of $\Sigma$. \\

Let $v \in C^\infty(M)$ be the lapse function of the foliation with respect to the parametrization above. Thus 
\[
L_\Sigma v = \frac{d}{d \, r} H_m (r) =  - 8 \,   e^{- 2 \, r} + 48 \, m \, e^{- 3 \, r} - 16\,  e^{- 4 \, r} + O (e^{- 5 \, r})
\] 
where 
\[
L_\Sigma = - \Delta_\Sigma - (\Ric (\nu, \nu) + |h|^2)
\]
is the stability operator of $\Sigma$. Note that  
\[
w = v - \fint_\Sigma v
\]  
satisfies the equation  
\[
- \Delta_\Sigma w  - \left ( \Ric (\nu, \nu) + |h|^2 \right) w = O (e^{- 5 \, r}).
\]
Moreover, 
\[
\fint_\Sigma v = 1 + O (e^{- 3 \, r}).
\]
Analyzing the spectrum of $L_\Sigma$ as in Lemma 3.13 of \cite{Huisken-Yau:1996} or Section 8 of \cite{Neves-Tian:2009}, we obtain 
\begin{align*}
\int_\Sigma w^2 &= O (e ^{- 2 \, r})\\
\int_\Sigma |\nabla_\Sigma w|^2 &= O (e^{- 4 \, r}). 
\end{align*}
Indeed, the distance of $0$ from the spectrum of $L_\Sigma$ is at least $48 \, m \, e^{- 3 \, r} (1 + o (1))$. 

It follows in particular that, as $r \to \infty$, 
\[
v = 1 + o (1).
\]
In our application below, it seems more natural to work with the lapse function $u \in C^\infty(\Sigma)$ for the original parametrization by the area $A$ of the canonical foliation. Note that 
\begin{align} \label{eqn:lapsefunction}
\int_\Sigma u = H^{-1} \qquad \text{ and } \qquad  \int_\Sigma |\nabla_\Sigma u|^2 = O (A^{-4}).
\end{align}

\begin {proposition} \label{lem:Hacf}
Let $H_A$ be the mean curvature of the the leaf $\Sigma_A$ in the canonical foliation. The Hawking mass along the foliation,  
\[
A \mapsto F(A) = m_{H} (\Sigma_A) = \frac{A^{\frac{1}{2}}}{(16 \, \pi)^{\frac{3}{2}}} \left( 16 \, \pi - (H_A^2 - 4) A\right),
\]
is continuously differentiable and 
\[
F (A) = m + O (A^{- 1}) \qquad \text{ and } \qquad F' (A) = O (A^{- 2}).
\] 
\begin {proof} 
From a standard computation using the first and second variation of area along with the Gauss equation, we obtain 
\begin{align*}
& (16 \, \pi)^{\frac{3}{2}} \, F' (A) \, q^2 \, A^{\frac{1}{2}}\\
& \qquad =  \int_\Sigma (R + 6) +   \int_\Sigma |\mathring h|^2 + 2 \int_\Sigma |\nabla_\Sigma \log u|^2 + \frac{1}{2} \, (q^2 - 1)  \, (16 \,  \pi  + 12 \, A - 3 \, A \, H^2) \\
& \qquad = O (A^{-\frac{3}{2}}) + O (A^{-2}) + 2 \, \int_\Sigma |\nabla_\Sigma \log u|^2 +  (q^2 - 1) \, O (A^{- \frac{1}{2}})
\end{align*}
where 
\[
q^2  = \fint_\Sigma u^{-1} \, \fint_\Sigma u .
\]
Estimate  \eqref{eqn:lapsefunction} for the lapse function gives 
\[
\int_\Sigma |\nabla_\Sigma \log u|^2 = O (A^{-2}) \qquad \text{ and } \qquad q^2 = 1 + O (A^{-2}),
\]
from which the assertion follows.
\end {proof}
\end {proposition}

\begin {corollary} As $A_1, A_2 \to \infty$, 
\begin{align} \label{HestimateNTspheres}
H_{A_1}^2 - H_{A_2}^2 = \left( 16 \, \pi + o (1) \right)  \, \left( A_1^{-1} - A_2^{-1} \right).  
\end{align} 
\end {corollary}


\section {Isoperimetric profile} \label{sec:isoperimetricprofile}

For convenient reference, we collect several properties of the isoperimetric profile of asymptotically hyperbolic Riemannian $3$-manifolds $(M, g)$. For simplicity, we assume that $\partial M$ is connected and the only closed $H = 2$ surface in $(M, g)$.  \\

First, recall the definition of the \textit{isoperimetric profile}  
\[
V_g : (\area_g(\partial M), \infty) \to (0, \infty)
\]
given by 
\[
A \mapsto V_g(A) = \sup \{ \vol_g(\Omega) : \Omega \in \mathcal{F}_A \} 
\] 
where
\[
\mathcal{F}_A = \{\text{$\Omega$ a compact region in $M$ with $\partial M \Subset \Omega$ and  $\area_g(\partial \Omega) = A$}\}.
\]
A region $\Omega \in \mathcal{F}_A$ with 
\[
V_g(A) = \area_g(\partial \Omega).
\] 
is called an \textit{isoperimetric region} for area $A$, and its boundary an \textit{isoperimetric surface}. 

It is well known and discussed in e.g.~\cite[p.~24]{Bray:1997} or \cite[p.~428]{Chodosh:2016} (see also \cite{Nardulli:2014}) that $V_g (A)$ is absolutely continuous and that for every $A > 0$, the left derivative $V_g(A)_-'$ and the right derivative $V_g(A)_+'$ exist, and that   
\begin {align} \label{eqn:convexity}
V_g (A)_-' \leq H^{-1} \leq V_g (A)_+'
\end {align}
where $H$ is the mean curvature of any isoperimetric surface of area $A$. 

Using the assumption on $\partial M$ and standard arguments, it follows that the isoperimetric profile is strictly increasing. From this, we see that the complement of an isoperimetric region has no bounded components. In particular, the boundary of a component of an  isoperimetric region has either two or one component, depending on whether it includes the horizon or doesn't.

By Theorem 1.1 in \cite{Chodosh:2016}, if we assume in addition that $R \geq - 6$, then isoperimetric regions for area $A$ exist provided that $A > \area_g(\partial M)$ is sufficiently large. 


\section {Some generalities about large isoperimetric regions}

For convenient reference, we collect several generalities about large isoperimetric regions in asymptotically hyperbolic manifolds. 

\begin {lemma} [Cf.~Lemma 2.2 in \cite{Ji-Shi-Zhu:2015}]
Let $(M, g)$ be a complete Riemannian $3$-manifold that is asymptotically hyperbolic. There is a constant $C >0 $ with the following property. For every isoperimetric surface $\Sigma \subset M$ in $(M, g)$, we have 
\begin {align*}
\area_g(\Sigma \cap B_r) \leq C \, e ^{- 2 \, r}.
\end {align*}
In particular, for every $p > 2$, 
\begin{align} \label{decayisop}
\int_\Sigma e^{- p \, r} = O (1).
\end{align}
\end {lemma}

\begin {lemma} 
Let $(M, g)$ be a complete Riemannian $3$-manifold that is asymptotically hyperbolic.  As $A \to \infty$,  
\begin{align} \label{eqn:largeisoH}
H = 2 + o (1)
\end{align}
where $H$ is the mean curvature of an isoperimetric surface $\Sigma$ with area $A$.
\begin {proof} 
The assertion follows from \eqref{CYR} and \eqref{decayisop}.
\end{proof}
\end {lemma}

The following result has been obtained in \cite{Chodosh:2016}. We include an alternative, more elementary derivation below. 

\begin {lemma} [Cf.~Proposition 6.4 of \cite{Chodosh:2016}] \label{lem:onelargecomponent}
Let $(M, g)$ be a complete Riemannian $3$-manifold that is asymptotically hyperbolic. There is a constant $A_0 > 1$ with the following property. Every isoperimetric region for area $A \geq A_0$ has a unique component $\Omega$ with $\area_g (\partial \Omega) \geq A_0$. Moreover, $(\partial \Omega) \setminus (\partial M)$ is connected.
\begin {proof} 
Suppose that there are two components $\Omega_i$ with $A_i = \area_g(\partial \Omega_i) \geq 1$ where $i = 1, 2$. For definiteness, let us assume that $A_1 \leq A_2$. Let $x : \R^3 \setminus B_1(0) \to M$ be a chart at infinity of $(M, g)$. We may choose regions $\bar \Omega_i \subset \R^3$ such that 
\[
x^{-1} ( \Omega_i ) \cup B_2(0) \subset \bar \Omega_i
\] 
and
\[
\area_g(\partial \Omega_i) = \area_{\bar g} (\partial \bar \Omega_i) + O (1) \qquad \text{ and } \qquad \vol_g (\Omega_i) = \vol_{\bar g} ( \bar \Omega_i) + O (1)
\]
where we have used the previous lemma for the area estimate. By the hyperbolic isoperimetric inequality, 
\[
\vol_{\bar g} ( \bar \Omega_i) \leq \frac{1}{2} \, A_i - \pi \, \log A_i + O (1).
\]
Using that $\Omega_1, \Omega_2$ are components of an isoperimetric region and that $(M, g)$ is asymptotically hyperbolic, we see that $\vol_{g} (\Omega_1) + \vol_g(\Omega_2)$ is at least as large as the volume of a ball of area $A_1 + A_2 - \area_g (\partial M)$ in hyperbolic space. It follows that
\[
\frac{1}{2} \left( A_1 + A_2 \right) - \pi \, \log (A_1 + A_2) + O (1) \leq \vol_{g} (\Omega_1) + \vol_g(\Omega_2). 
\] 
From this and the previous estimate, we obtain
\[
\frac{1}{2} \left( A_1 + A_2 \right) - \pi \, \log (A_1 + A_2) + O (1) \leq \Big( \frac{1}{2}\, A_1 - \pi \, \log A_1+ O (1) \Big) +  \Big( \frac{1}{2}\, A_2 - \pi \, \log A_2+ O (1) \Big) 
\]
Put another way, 
\[
\frac{1}{2} \, A_1 \leq \frac{A_1 \, A_2}{A_1 + A_2} \leq O (1).
\]
The connectedness of the outer boundary follows from the monotonicity of the isoperimetric profile at infinity; see Lemma 3.5 in \cite{Chodosh:2016}. 
\end {proof}
\end {lemma}

\begin {lemma} \label{lem:smallcomponents}
Let $(M, g)$ be a complete Riemannian $3$-manifold that is asymptotically hyperbolic. Let $\{\Omega_i\}_{i = 1}^\infty$ be a sequence of isoperimetric regions in $(M, g)$ with $\area_{g}(\partial \Omega_{i})\to \infty$. Let $\Sigma_i$ be a component of $\partial \Omega_i$ such that $\area_{g} (\Sigma_i) < A_0$ where $A_0 > 1$ is as in Lemma \ref{lem:onelargecomponent}. The distance between $\Sigma_i$ and $\partial M$ tends to zero as $i \to \infty$. 
\begin {proof} 
We have $H_i = 2 + o (1)$ for the mean curvature of $\Sigma_i$ by \eqref{eqn:largeisoH}. The diameter of $\Sigma_i$ is a priori bounded by the monotonicity formula. If the sequence has a subsequential limit in $M$, then this limit is a closed surface of constant mean curvature $2$ and hence a component of $\partial M$. If the sequence is divergent, then we can follow a subsequence (in the sense of pointed geometric convergence) to a closed surface of constant mean curvature $2$ in hyperbolic space. Such surfaces do not exist.
\end {proof}
\end {lemma}


\section {Extensions of results from \cite{Chodosh:2016}} 

In this section, we collect several extensions of results in the work of the first-named author \cite{Chodosh:2016} to the case where $(M, g)$ is asymptotic to Schwarzschild-anti-deSitter, rather than exactly Schwarzschild-anti-deSitter outside of a compact set. \\

\begin {lemma} [Cf.~the proof of ``Case 3" in Theorem 1.1 in \cite{Chodosh:2016} and Proposition 3.1 in \cite{Ji-Shi-Zhu:2015}] \label{lem:nodrift}
Let $(M, g)$ be a complete Riemannian $3$-manifold with $R \geq -6$ that is asymptotically hyperbolic, but not hyperbolic space, and such that $\partial M$ is connected and the only closed $H = 2$ surface in $(M, g)$. There are $A_0 > 1$ and $r_0 > 1$ with the following property. Let $\Omega$ be the unique large component of an isoperimetric region $\tilde \Omega$ for area $\tilde A \geq A_0$.  Then 
\[
\Omega \cap B_{r_0} \neq \emptyset.
\]
\begin {proof}
Assume that $\Omega \cap B_{r} = \emptyset$ where $r >1$ is large. As $r \to \infty$, 
\begin{align*}
\vol_g(\tilde \Omega) &= \vol_{g} ( \Omega) + o (1), \\
\area_g(\partial \tilde \Omega) &= \area_g (\partial \Omega) + \area_g (\partial M) + o (1),
\end{align*}
where we have used Lemma \ref{lem:onelargecomponent} and Lemma \ref{lem:smallcomponents}. 
Similarly,  
\begin{align*}
\vol_{\bar g} ( \Omega) &= \vol_{g} ( \Omega) + o (1), \\
\area_{\bar g}(\partial  \Omega) &= \area_{g} (\partial  \Omega) + o (1).
\end{align*}
Let $A = \area_g (\partial  \Omega)$. By the hyperbolic isoperimetric inequality, as $A \to \infty$, 
\begin{align*}
\vol_{\bar g} ( \Omega) \leq \frac{1}{2} \, A- \pi \, \log A+ \pi \, (1 + \log \pi) + o (1).
\end{align*}
Using that $\tilde \Omega$ is isoperimetric and Lemma \ref{lem:isoballs}, as $A \to \infty$, 
\begin{align*}
& \vol_{g} (\tilde \Omega) \\
&\geq \vol_g (\text{centered coordinate ball with the same boundary area as } \partial \tilde \Omega)\\
&\geq \frac{1}{2} \, \big( A+ \area_g(\partial M) + o (1) \big) - \pi \, \log \big(A + \area_g(\partial M) + o (1) \big) +  \pi \, (1 + \log \pi) + V(M, g) + o (1) \\
& = \frac{1}{2} \, A - \pi \, \log A+  \pi \, (1 + \log \pi) + \underbrace{ V(M, g) + \frac{1}{2} \area_g(\partial M) }_{(\ast)} + o (1). 
\end{align*}
The quantity $(\ast)$ is positive by Proposition 5.3 in \cite{Chodosh:2016} (building on the earlier work \cite{Brendle-Chodosh:2014} of the first-named author with S.\ Brendle). These estimates are not compatible. 
\end {proof}
\end {lemma} 

\begin {remark} 
We expect that the assumption that $\partial M$ be connected in Lemma \ref{lem:nodrift} can be removed by using the inverse mean curvature flow with forced jumps along with computations as in \cite[p.\ 427]{Chodosh:2016}.
\end {remark}

\begin {lemma} [Cf.~Proposition 8.3 in \cite{Chodosh:2016}] \label{lem:mHbounds}
Let $(M, g)$ be a complete Riemannian $3$-manifold that is asymptotic to Schwarzschild-anti-deSitter with mass $m > 0$. We also assume that $R \geq - 6$ and that $\partial M$ is connected and the only closed $H = 2$ surface in $(M, g)$. There is $A_0 > 1$ with the following property. Assume that $\Omega$ is the unique large component of an isoperimetric region for area $A \geq A_0$. Let $\Sigma = (\partial \Omega) \setminus (\partial M)$. Then $\Sigma$ is connected and 
\[
m_H (\Sigma) \leq 4 \, m.
\]
\begin {proof}
We describe the minor modifications to the proof of Proposition 8.3 in \cite{Chodosh:2016}, where the same result is shown under the additional assumption that $(M, g)$ is equal to Schwarzschild-anti-deSitter outside of a compact set. We use Lemma \ref{lem:isoballs} below instead of Lemma A.2 in \cite{Chodosh:2016}. We use Lemma \ref{lem:nodrift} instead of applying S.~Brendle's characterization of closed constant mean curvature surfaces in exact Schwarzschild-anti-deSitter in \cite[p.~427]{Chodosh:2016}. 
\end {proof}
\end {lemma}

\begin {lemma} [Cf.~Lemma A.2 in \cite{Chodosh:2016}] \label{lem:isoballs}
Let $(M, g)$ be a complete Riemannian $3$-manifold that is asymptotic to Schwarzschild-anti-deSitter with mass $m > 0$. For $A > 0$ large, let $r (A) > 0$ be such that the centered coordinate sphere $S_{r(A)}$ has area $A$. Denoting the renormalized volume of $(M, g)$ by $V(M, g)$, we have  the expansion
\begin{align*}
\vol_g(B_{r(A)}) = \frac{1}{2} \, A - \pi \, \log A + \pi \, (1 + \log \pi) + V(M, g) - 8 \, \pi^{\frac{3}{2}} \, m \, A^{- \frac{1}{2}}  + O (A^{-1}).
\end{align*}
\begin {proof}
The proof of Lemma A.2 in \cite{Chodosh:2016} for $(M, g)$ equal to Schwarzschild-anti-deSitter outside of a compact set extends to the present generality. 
\end {proof}
\end {lemma}


\bibliography{bib} 
\bibliographystyle{amsplain}
\end{document}